\documentclass{amsart}
\usepackage[utf8]{inputenc}
\usepackage{amsmath}
\usepackage{amssymb}
\usepackage{caption}
\usepackage{amsthm}
\usepackage[hidelinks]{hyperref}
\usepackage{cleveref}
\crefname{lemma}{Lemma}{Lemmas}
\crefname{theorem}{Theorem}{Theorems}
\usepackage{xcolor}
\usepackage{comment}
\usepackage{graphicx}
\usepackage{tikz}
\usepackage{tikz-cd}
\usepackage{mathrsfs}
\usetikzlibrary{shapes.geometric}
\makeatletter
\def\@settitle{\begin{center}%
  \baselineskip14\p@\relax
  \bfseries
  \uppercasenonmath\@title
  \@title
  \ifx\@subtitle\@empty\else
     \\[1ex]\uppercasenonmath\@subtitle
     \footnotesize\mdseries\@subtitle
  \fi
  \end{center}%
}
\def\subtitle#1{\gdef\@subtitle{#1}}
\def\@subtitle{}
\makeatother

\newtheorem{theorem}{Theorem}[section]

\newtheorem{lemma}[theorem]{Lemma}

\theoremstyle{remark}
\newtheorem{remark}[theorem]{Remark}

\theoremstyle{remark}

\usepackage{chngcntr}
\usepackage{graphicx} 
\usepackage{float}
\counterwithout{equation}{section}
\counterwithout{theorem}{section}

\begin{document}

\title{A novel witness to incoherence of $\mathrm{SL}_5(\mathbb{Z})$}

\author{Sami Douba}

\begin{abstract}
Motivated by a question of Stover, we discuss an example of a Zariski-dense finitely generated subgroup of $\mathrm{SL}_5(\mathbb{Z})$ that is not finitely presented.
\end{abstract}

\address{Institut des Hautes \'Etudes Scientifiques, 
Universit\'e Paris-Saclay, 35 route de Chartres, 91440 Bures-sur-Yvette, France}

\maketitle

A group $\Gamma$ is {\em coherent} if all its finitely generated subgroups are finitely presented. Otherwise, we say $\Gamma$ is {\em incoherent}. In this note, we will concern ourselves with the case that $\Gamma$ is a discrete subgroup of a noncompact simple Lie group. For a broader discussion on coherence, we refer the reader to Wise's survey \cite{wise2020invitation}. 
%If a group property $\mathcal{P}$ is inherited by subgroups and if every finitely generated group with property $\mathcal{P}$ is finitely presented, then all groups satisfying $\mathcal{P}$ are coherent. Since the property of being polycyclic, that of being free, and that of being the fundamental group of a (possibly open) surface are all examples of such properties $\mathcal{P}$, one obtains 

Since finitely generated fundamental groups of (possibly open) surfaces are finitely presented, all Fuchsian groups are coherent. That the former statement holds also for $3$-manifolds is due independently to Scott \cite{scott1973finitely, scott1973compact} and Shalen (unpublished); in particular, all Kleinian groups are in fact coherent. On the other hand, the product $F_2 \times F_2$, where $F_2$ is a free group of rank two, admits a map onto $\mathbb{Z}$ whose kernel is finitely generated but not finitely presented~\cite[Example~9.22]{wise2020invitation}. It follows that $\mathrm{SL}_n(\mathbb{Z})$ is incoherent for $n \geq 4$, since the latter possesses a block-diagonal copy of $F_2 \times F_2$. Whether or not $\mathrm{SL}_3(\mathbb{Z})$ is coherent is a question of Serre \cite[Problem~F14]{wall1979homological} and remains open.

While a simple Lie group $G$ of real rank one possesses no discrete copies of $F_2 \times F_2$, many (conjecturally, all) lattices in $G$ are incoherent as soon as $G$ is not locally isomorphic to $\mathrm{SL}_2(\mathbb{R})$ or $\mathrm{SL}_2(\mathbb{C})$; see Kapovich–Potyagailo–Vinberg \cite{kapovich2008noncoherence} and Kapovich~\cite{kapovich2013noncoherence}. 
%whose arguments exploit the phenomenon that many\footnote{It is now known by the resolution of Thurston's virtually fibered conjecture that every real hyperbolic lattice in dimension $3$ virtually admits a map onto $\mathbb{Z}$ with finitely generated kernel; for a history of this remarkable result, see \cite[Ch.~4]{aschenbrenner20153}. As for the complex hyperbolic setting, it seems it remains open whether every arithmetic lattice in $\mathrm{SU}(2,1)$ virtually admits a map onto $\mathbb{Z}$ at all, or even whether every such lattice lacks the congruence subgroup property.} real hyperbolic lattices in real dimension~$3$, and complex hyperbolic lattices in complex dimension~$2$, admit maps onto $\mathbb{Z}$ with finitely generated kernels. 
We remark that, when ${n \geq 5}$, another way of certifying that $\mathrm{SL}_n(\mathbb{Z})$ is incoherent is to observe that $\mathrm{SL}_n(\mathbb{Z})$ contains a copy of the incoherent group $\mathrm{SO}(4,1; \mathbb{Z})$; indeed, as observed in~\cite{kapovich2008noncoherence}, the first construction of a geometrically finite incoherent subgroup of $\mathrm{SO}(4,1)$, due to Kapovich and Potyagailo \cite{kapovich1991absence}, can be carried out within $\mathrm{SO}(4,1; \mathbb{Z})$. 

We call a finitely generated but not finitely presented subgroup of a group $\Gamma$ a {\em witness to incoherence of $\Gamma$}.  A question posed by Stover \cite{stover2019coherence} that is not addressed by the above discussion is whether there are {\em Zariski-dense} witnesses to incoherence of $\mathrm{SL}_n(\mathbb{Z})$, $n \geq 4$. It was suggested to the author by K. Tsouvalas that, if one allows freely decomposable examples, then Stover's question has an affirmative answer at least for $n \geq 5$. Indeed, a ping-pong argument in the projective space~$\mathbb{P}(\mathbb{R}^n)$ demonstrates that $\mathrm{SL}_n(\mathbb{Z})$ contains for each $n \geq 5$ a subgroup decomposing as $\Gamma_0 * F$, where $\Gamma_0 < \mathrm{SO}(4,1; \mathbb{Z})$ is a geometrically finite incoherent group arising from the Kapovich–Potyagailo construction and $F$ is a Zariski-dense Anosov free subgroup of $\mathrm{SL}_n(\mathbb{Z})$ \cite{douba2023regular}. For any witness $\Delta_0 < \Gamma_0$ to incoherence of $\Gamma_0$, the subgroup $\langle \Delta_0, F \rangle < \mathrm{SL}_n(\mathbb{Z})$ is then a Zariski-dense witness to incoherence of $\mathrm{SL}_n(\mathbb{Z})$. The purpose of this note is to describe a Zariski-dense witness to incoherence of $\mathrm{SL}_5(\mathbb{Z})$ of a different nature; for instance, this witness has cohomological dimension $3$, its limit set in $\mathbb{P}(\mathbb{R}^5)$ in the sense of Guivarc'h \cite{guivarch1990produits} is connected (indeed, is the boundary of an exotic properly convex projective domain), and we suspect this witness is one-ended.

\tikzset{my dbl/.style={double,double distance=2pt}}
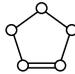
\begin{figure}
\scalebox{2}{ \begin{tikzpicture}[mystyle/.style={draw,shape=circle,fill=white, inner sep=0pt, minimum size=2pt, line width=0.01cm}]
\def\ngon{5}
\node[draw, regular polygon,regular polygon sides=\ngon,minimum size=0.4cm, line width=0.01cm] (p) {};
\draw[style=double,line width=0.01cm] (p.corner 3) -- (p.corner 4);
\foreach\x in {1,...,\ngon}{
    \node[mystyle=\x] (p\x) at (p.corner \x){};
}
\end{tikzpicture}}
\caption{The diagram $\Sigma$.}
\label{fig:diagram}
\end{figure}
We now proceed to the description. Let $W$ be the Coxeter group given by the diagram $\Sigma$ shown in Figure \ref{fig:diagram}.
The Cartan matrix $A$ of $\Sigma$ is
$$
A = \begin{pmatrix}
2 & -\sqrt{2} & 0 & 0 & -1 \\
-\sqrt{2} & 2 & -1 & 0 & 0 \\
0 & -1 & 2 & -1 & 0  \\
0 & 0 & -1 & 2 & -1 \\
-1 & 0 & 0 & -1 & 2
\end{pmatrix}.
$$
Since $A$ has signature $(4,1)$, the Coxeter group $W$ can be realized as a reflection group in $\mathrm{Isom}(\mathbb{H}^4)$, and this realization is unique up to conjugation, so that we may conflate $W$ with its image in $\mathrm{Isom}(\mathbb{H}^4)$; see Vinberg \cite[Thm.~2.1]{vinberg1985hyperbolic}. Moreover, since all proper induced subdiagrams of $\Sigma$ are elliptic, we have that each fundamental chamber for $W$ in $\mathbb{H}^4$ is a compact Coxeter simplex $P \subset \mathbb{H}^4$.

\begin{remark}\label{arith}
The compact hyperbolic Coxeter simplices were enumerated by Lann\'er \cite{lanner1950complexes}. The simplex $P$ is one among only five such simplices in dimension $4$, the highest dimension in which such simplices exist, and, crucially for us, is the only one among those with the property that, for each of its dihedral angles $\theta$, the quantity $4\cos^2\theta$ is an integer. 
We remark also that the Galois conjugate $A^\sigma$ of $A$ is positive-definite, where $\sigma: \mathbb{Q}(\sqrt{2}) \rightarrow \mathbb{Q}(\sqrt{2})$ is the unique nontrivial automorphism, so that $W$ is in fact an arithmetic subgroup of $\mathrm{Isom}(\mathbb{H}^4)$. 
\end{remark}

While it is unclear to us whether the simplex $P$ tiles a compact right-angled polyhedron in $\mathbb{H}^4$, it is nevertheless true that some finite-index subgroup of $W$ abstractly embeds in a right-angled Coxeter group. Indeed, the latter remains true when $W$ is replaced with an arbitrary finitely generated Coxeter group by work of Haglund and Wise \cite{haglund2010coxeter}; alternatively, one can exploit arithmeticity of $W$ (see Remark~\ref{arith}) and apply a result of Bergeron, Haglund, and Wise \cite{bergeron2011hyperplane}. In particular, the group $W$ is virtually residually finite rationally solvable (RFRS) in the sense of Agol \cite{agol2008criteria}, and hence, by a theorem of Kielak \cite{kielak2020residually}, possesses a torsion-free finite-index subgroup $\Lambda < W$, which we may assume preserves orientation, that maps onto $\mathbb{Z}$ with finitely generated kernel $\Delta \lhd \Lambda$. That $\Delta$ is not finitely presented is an application of the following lemma to $M = \Lambda \backslash \mathbb{H}^4$ and $\tilde{M} = \Delta \backslash \mathbb{H}^4$; for a more general statement, see Llosa Isenrich, Martelli, and Py \cite[Prop.~14]{isenrich2021hyperbolic}. Note that $\chi(\Lambda \backslash \mathbb{H}^4) > 0$ by the Chern–Gauss–Bonnet theorem. 

\begin{lemma}\label{milnor}
Let $M$ be a closed connected oriented aspherical $4$-manifold with $\chi(M) \neq 0$. Then $\pi_1(\tilde{M})$ is not finitely presented for any infinite cyclic cover~$\tilde{M}$ of $M$.
\end{lemma}

\begin{proof}
By a result of Milnor \cite{milnor1968infinite}, there is some $i \in \mathbb{N}$ such that $H_i(\tilde{M}; \mathbb{Q})$ has infinite dimension. Since $\tilde{M}$ is an open $4$-manifold, we have $b_i(\tilde{M}) = 0$ for $i > 3$. Now if $\pi_1(\tilde{M})$ is not finitely generated, then we are done. Otherwise, the dimension of $H_1(\tilde{M}; \mathbb{Q})$ is finite, and hence, by the partial duality theorem of Kawauchi \cite{kawauchi1975partial}, we have $H_3(\tilde{M}; \mathbb{Q}) = \mathbb{Q}$. This forces the dimension of $H_2(\tilde{M}; \mathbb{Q})$ to be infinite, so that $\pi_1(\tilde{M})$ cannot be finitely presented.
\end{proof}

Now let $\rho: W \rightarrow \mathrm{GL}_5(\mathbb{R})$ be the representation given by
\[
\rho(s_i)(v) = v - (v^T A' e_i)e_i
\]
for each $v \in \mathbb{R}^5$, where $s_1, \ldots, s_5 \in W$ are the standard generators of $W$, and
\[
A' = \begin{pmatrix}
2 & -1 & 0 & 0 & -1 \\
-2 & 2 & -1 & 0 & 0 \\
0 & -1 & 2 & -1 & 0  \\
0 & 0 & -1 & 2 & -1 \\
-1 & 0 & 0 & -1 & 2
\end{pmatrix}.
\]
Note that since $A'$ has integer entries, in fact $\rho(W) \subset \mathrm{GL}_5(\mathbb{Z})$. The representation~$\rho$ is the Vinberg representation of $W$ with Cartan matrix $A'$; since $A'$ is of negative type, we have in particular that $\rho$ is faithful and preserves a properly convex domain $\Omega \subset \mathbb{P}(\mathbb{R}^5)$; see Vinberg \cite{MR0302779}. As in the hyperbolic setting, since all proper induced subdiagrams of $\Sigma$ are elliptic, we have that $\rho(W)$ acts cocompactly on $\Omega$ (in the language of Benz\'ecri \cite{benzecri1960varietes}, the group $\rho(W)$ {\em divides} $\Omega$). The same is then true of $\rho(\Lambda)$. Since the matrix $A'$ is not symmetric, the domain $\Omega$ is not an ellipsoid.\footnote{In fact, no subgroup of $\mathrm{SL}_n(\mathbb{Z})$ divides an ellipsoid in $\mathbb{P}(\mathbb{R}^n)$ for $n \geq 5$. Indeed, such a subgroup $\Gamma < \mathrm{SL}_n(\mathbb{Z})$ would have to be of finite index in $\mathrm{SO}(Q; \mathbb{Z})$ for some rational quadratic form $Q$ in $n \geq 5$ variables, but such $Q$ is isotropic by Meyer's theorem, so that $\Gamma$ cannot be cocompact in $\mathrm{SO}(Q; \mathbb{R})$.} A theorem of Benoist \cite{benoist2003convexes} then asserts that $\rho(\Lambda)$ is Zariski-dense in $\mathrm{SL}_5(\mathbb{R})$. By simplicity of the latter, the infinite normal subgroup $\rho(\Delta) \lhd \rho(\Lambda)$ remains Zariski-dense in $\mathrm{SL}_5(\mathbb{R})$. We conclude that $\rho(\Delta)$ is a Zariski-dense witness to incoherence of $\mathrm{SL}_5(\mathbb{Z})$. 

\begin{remark}
The representation $\rho$ of $W$ also appears in Choi and Choi \cite[Prop.~11]{choi2015definability}, and is analogous to the triangle group representations into $\mathrm{GL}_3(\mathbb{Z})$ discovered by Kac and Vinberg \cite{vinberg1967quasi}. 
\end{remark}

\begin{remark}
The above strategy will not produce a Zariski-dense witness to incoherence of $\mathrm{SL}_4(\mathbb{Z})$, since any finitely generated subgroup of $\mathrm{GL}_4(\mathbb{R})$ preserving and acting properly on a domain in $\mathbb{P}(\mathbb{R}^4)$---for instance, any Vinberg reflection group in $\mathrm{GL}_4(\mathbb{R})$, hence any Coxeter group on at most $4$ vertices\footnote{That a Coxeter group on at most $4$ vertices is virtually a $3$-manifold group can also be deduced, by an argument along the lines of \cite[Example~1.2]{dani2023right}, from the fact that any non-elliptic Coxeter system on at most $4$ vertices has planar nerve.}---is virtually the fundamental group of a $3$-manifold and is thus coherent. However, it seems plausible that suitable Vinberg representations of incoherent Coxeter groups might yield one-ended Zariski-dense witnesses to incoherence of $\mathrm{SL}_n(\mathbb{Z})$ for arbitrarily large $n$. The approach that appears most promising in this regard is to consider Vinberg representations of finite-index reflection subgroups $W_n < W$ of increasing rank, where~$W$ is a fixed Coxeter group that virtually maps onto $\mathbb{Z}$ with finitely generated but not finitely presented kernel. However, such representations will not divide projective domains in high dimensions, so different methods of certifying Zariski-density would be required.
\end{remark}

\subsection*{Acknowledgements} We are very grateful to Balthazar Fl\'echelles, Gye-Seon Lee, Pierre Py, and Konstantinos Tsouvalas for helpful discussions. The author was supported by the Huawei Young Talents Program.

\bibliography{bantingbib}{}
\bibliographystyle{siam}

\end{document}